\documentclass[11pt]{amsart}

\usepackage{bm}
\usepackage{graphicx}
\usepackage[all]{xy}
\usepackage{mathpazo}
\usepackage{latexsym}   
\usepackage{amssymb}    
\usepackage{amsmath}    
\usepackage{amsthm}     
\usepackage{hyperref}
\usepackage{tikz}

\usepackage[margin=1in]{geometry}

\renewcommand{\bar}[1]{{\overline{#1}}}

\numberwithin{equation}{section}
\numberwithin{figure}{section}

\theoremstyle{plain}
\newtheorem{theorem}{Theorem}[section]
\newtheorem{lemma}[theorem]{Lemma}
\newtheorem{proposition}[theorem]{Proposition}
\newtheorem{corollary}[theorem]{Corollary}

 \newtheorem{question}[theorem]{Question}
\theoremstyle{definition}
\newtheorem{definition}[theorem]{Definition}
\newtheorem{remark}[theorem]{Remark}
\newtheorem{example}[theorem]{Example}

\newcommand{\stk}[1]{{\mathcal #1}}
\newcommand{\tilstk}[1]{{\til{\stk #1}}}
\newcommand{\barstk}[1]{{\bar{\stk #1}}}

\def\setcomp{\smallsetminus}

\def\kk{{\mathbb K}}
\def\integ{{\mathbb Z}}
\def\iso{\cong}
\def\ff{{\mathbb F}}
\global\let\hom\undefined
\DeclareMathOperator{\hom}{Hom}
\def\mmu{{\pmb\mu}}
\def\rat{{\mathbb Q}}
\newcommand{\st}[1]{\{#1\}}
\DeclareMathOperator{\spec}{Spec}
\DeclareMathOperator{\codim}{codim}
\def\ra{\rightarrow}
\DeclareMathOperator{\jac}{Jac}
\newcommand{\oneover}[1]{\frac{1}{#1}}
\newcommand{\half}[1]{\frac{#1}{2}}
\def\del{\partial}
\newcommand{\til}[1]{{\widetilde{#1}}}
\def\inject{\hookrightarrow}
\def\cx{{\mathbb C}}
\newcommand{\abs}[1]{{\left|#1\right|}}
\newcommand{\floor}[1]{{\lfloor #1 \rfloor}}
\def\cross{\times}
\DeclareMathOperator{\pic}{Pic}
\def\gp{{\mathbb G}}
\newenvironment{alphabetize}{\begin{enumerate}

}{\end{enumerate}}
\def\calm{{\mathcal M}}
\def\cala{{\mathcal A}}

\begin{document}

\title{Generic Newton polygons for curves of given $p$-rank}

\author{Jeff Achter}
\address{Colorado State University, Fort Collins, CO 80523}
\email{achter@math.colostate.edu}
\urladdr{http://www.math.colostate.edu/~achter}

\author{Rachel Pries}
\address{Colorado State University, Fort Collins, CO 80523}
\email{pries@math.colostate.edu}
\urladdr{http://www.math.colostate.edu/~pries}

\begin{abstract}
We survey results and open questions about the $p$-ranks and Newton polygons of Jacobians of curves in positive characteristic $p$.  
We prove some geometric results about the $p$-rank stratification of the moduli space of (hyperelliptic) curves.
For example, if $0 \leq f \leq g-1$, 
we prove that every component of the $p$-rank $f+1$ stratum of $\calm_g$ contains a component of the $p$-rank $f$ stratum in its closure.
We prove that the $p$-rank $f$ stratum of $\overline{\calm}_g$ is connected.
For all primes $p$ and all $g \geq 4$, 
we demonstrate the existence of a Jacobian of a smooth curve, defined over $\bar\ff_p$,  whose Newton polygon has slopes $\{0, \frac{1}{4}, \frac{3}{4}, 1\}$.
We include partial results about the generic Newton polygons of curves of given genus $g$ and $p$-rank $f$.
\end{abstract}

%
\thanks{
The first author is supported in part by Simons Foundation grant 204164. 
The second author is supported in part by NSF grant DMS-11-01712.}


\maketitle

\section{Introduction}

Suppose $C$ is a smooth projective curve of genus $g$ defined over a finite field $\ff_q$ of characteristic $p$.
Then its zeta function has the form $Z_{C/\ff_q}(T)=\frac{L_{C/\ff_q}(T)}{(1-T)(1-qT)}$ for some polynomial $L_{C/\ff_q}(T) \in {\mathbb Z}[T]$.
The Newton polygon $\nu$ of $C$ is that of $L_{C/\ff_q}(T)$; 
it is a lower convex polygon in ${\mathbb R}^2$ with endpoints $(0,0)$ and $(2g,g)$.
Its slopes encode important information about $C$ and its Jacobian.

Given a curve $C/\ff_q$ of genus $g$, there are methods to compute its Newton polygon.  
After some experiments, it becomes clear that the typical Newton polygon has slopes only $0$ and $1$.
For small $g$ and $p$, the other possible Newton polygons do occur, but rarely, leading us to the following question.

\begin{question}
Does every Newton polygon of height $2g$ (satisfying the obvious necessary conditions)
occur as the Newton polygon of a smooth curve defined over a finite field of characteristic $p$ for each prime $p$?
\end{question}

The answer to this question is unknown, although one now knows that every integer $f$ such that $0 \leq f \leq g$ occurs as the
length of the line segment of slope $0$ for the Newton polygon of a
curve in each characteristic $p$ \cite{fvdg}.  As an example, we
consider the first open case, when $g=4$ and $\nu$ has slopes
$\frac{1}{4}$ and $\frac{3}{4}$.  We confirm in Lemma \ref{lemg4} that
this Newton polygon occurs for a curve of genus $4$ for each prime $p$
using a unitary Shimura variety of type $U(3,1)$.

The main idea in this paper is that the occurrence of a certain Newton polygon for a curve of small genus
can be used to prove the occurrence of new Newton polygons for smooth curves for every larger genus.  
As an application, we prove in Corollary \ref{thmgnpg4}
that the Newton polygon $\nu_g^{g-4}$ having $g-4$ slopes of $0$ and $1$ and 
four slopes of $\frac{1}{4}$ and $\frac{3}{4}$ occurs as the Newton polygon of a smooth curve of genus $g$
for all primes $p$ and all $g \geq 4$.

The key condition above is that the curve must be smooth, 
because it is easy to produce singular curves with decomposable Newton polygons by 
clutching together curves of smaller genus.
In order to deduce results about Newton polygons of smooth curves from results about Newton polygons of singular curves, 
we rely on geometric methods from \cite{achterpriesprank}.
It turns out that one of the best techniques to determine the existence of a curve whose Jacobian has specified behavior 
is to study the geometry of the corresponding loci in $\stk M_g$, the moduli space of smooth proper curves of genus $g$.  

More precisely, the $p$-rank $f$ and Newton polygon are invariants of
the $p$-divisible group of a principally polarized abelian variety.
The stratification of the moduli space $\cala_g$ by these invariants
is well-understood, in large part because of work of Chai and Oort.  Let $\stk A_g$ be the moduli space of principally polarized abelian varieties of dimension $g$.  
The Torelli map $\tau: \stk M_g \inject \stk A_g$, which sends a curve
to its Jacobian, allows us to define the analogous stratifications on
$\calm_g$.  For dimension reasons, this gives a lot of information
when $1 \leq g \leq 3$ and very little information when $g \geq 4$.
For example, in most cases it is not known whether the $p$-rank $f$
stratum $\calm_g^f$ is irreducible.

In Section \ref{secdefinitions}, we review the fundamental definitions
and properties of the $p$-rank and Newton polygon.    In Section \ref{secag}, we review the $p$-rank and Newton polygon stratifications of $\stk A_g$.
Since degeneration
is one of the few techniques for studying stratifications in $\stk
M_g$, in Section \ref{subsecdm} we recall the Deligne-Mumford
compactification of $\stk M_g$, and explain how it interacts with the
$p$-rank stratification.

In Section \ref{subsecmgbarf}, we review a theorem that we proved about the boundary of the $p$-rank strata $\calm_g^f$ of $\calm_g$ in \cite{achterpriesprank}.
Using this, we prove that $\overline{\calm}^f_g$ is connected for all $g \geq 2$ and $0 \leq f \leq g$ (Corollary \ref{thmgconn}).
For $f \geq 1$, we also prove that every component of $\calm_g^f$ contains a component of $\calm_g^{f-1}$ in its closure (Corollary \ref{PnestedMg}).

In Section \ref{secstratnp}, we consider the finer stratification of
$\calm_g$ by Newton polygon.  We consider a Newton polygon $\nu_g^f$
which is the most generic Newton polygon of an abelian variety of
dimension $g$ and $p$-rank $f$.  The expectation is that
the generic point of every component of $\calm_g^f$
represents a curve with Newton polygon $\nu_g^f$.  We prove that this expectation holds 
in the first non-trivial case when $f=g-3$ in Corollary
\ref{cormgnpg3} and prove a slightly weaker statement when $f=g-4$
in Corollary \ref{thmgnpg4}.

The discrete invariants associated to these stratifications seem to
influence arithmetic attributes of curves over finite fields, such as
automorphism groups and maximality.  One should note, however, that
this relationship is somewhat subtle.
On one hand, many exceptional curves turn out to be {\em
supersingular}, meaning that the Newton polygon is a line segment of slope $\frac{1}{2}$.
For example, it is not hard to prove that a curve which achieves the Hasse-Weil bound over a finite field
must be supersingular.
On the other hand, the $p$-rank stratification is in some ways ``transverse'' to
other interesting loci in $\stk M_g$, illustrated by the fact that a randomly chosen
Jacobian of genus $g$ and $p$-rank $f$ behaves like a randomly
selected principally polarized abelian variety of dimension $g$.
In  Section \ref{subsecopenprank} and Section \ref{secstratnp} we discuss  open questions and conjectures on these topics.



We thank the organizers for the opportunity to participate in  the RICAM workshop on algebraic curves over finite fields.

\section{Structures in positive characteristic}
\label{secdefinitions}

Consider a principally polarized abelian variety $X$ of dimension $g$ defined over a field
$K$ of characteristic $p>0$.  If $N\ge 2$ is relatively prime to $p$,
then  the $N$-torsion group scheme $X[N]$ is \'etale, and 
$X[N](\bar K) \iso (\integ/N)^{\oplus 2g}$ depends only on the
dimension of $X$.  In contrast, $X[p]$ is {\em never} reduced, and
there is a range of possibilities for the geometric isomorphism class
of $X[p]_{\bar K}$ and, {\em a fortiori}, the $p$-divisible group
$X[p^\infty] := \lim_{\ra n} X[p^n]$.  In this section, we review some
attributes of $X[p]$ and $X[p^\infty]$, with special emphasis on the
case where $X$ is the Jacobian of a curve over a finite field.  

\subsection{The $p$-rank}
\label{subsecprank}

The $p$-rank of $X$ is the rank of the ``physical'' $p$-torsion of
$X$.  More precisely, it is the integer
$f$ such that 
\begin{equation}
\label{eqdefprank}
X[p](\bar K) \iso (\integ/p)^{\oplus f}.
\end{equation}

We will see below (\ref{subsubsecnparbitrary}) that $0 \le f \le g$.  The abelian variety $X$ is
said to be ordinary if its $p$-rank is maximal, i.e., $f=g$.  

Specifying a $K$-point of
$X[p]$ is equivalent to specifying a homomorphism $X[p] \ra (\integ/p)$
of group schemes over $K$, and thus one may also define $f$ by
\[
f = \dim_{\ff_p}\hom_{\bar K}(X[p], (\integ/p)).
\]
Now, $X[p]$ is a self-dual group scheme, and the dual of $(\integ/p)$
is the non-reduced group scheme $\mmu_p$, the kernel of Frobenius on the multiplicative group ${\mathbb G}_m$.  
Consequently, it is
equivalent to define the $p$-rank of $X$ as
\[
f = \dim_{\ff_p}\hom_{\bar K}(\mmu_p, X[p]).
\]
(This last formulation is convenient for defining the $p$-rank of semiabelian varieties and semistable curves.)

If $X$ is the Jacobian of a smooth, projective curve $C$, then the $p$-rank equals the maximum
rank of a $p$-group which occurs as the Galois group of an unramified cover of $C$
\cite[Corollary 4.18]{milne}.

\subsection{Newton polygons}

\subsubsection{Newton polygon of a curve over a finite field}
\label{subsubnpcurve} 

Let $C/\ff_q$ be a smooth, projective curve of genus $g$.  Then its
zeta function 
\[Z_{C/\ff_q}(T) = \exp(\sum_{k \ge 1} \#C(\ff_{q^k}) T^k/k)\]
is a rational function of the form
\[Z_{C/\ff_q} = \frac{L_{C/\ff_q}(T)}{(1-T)(1-qT)}\]
where $L_{C/\ff_q}(T) \in \integ[T]$ is a polynomial of degree $2g$.  The $L$-polynomial factors
over $\bar\rat$ as 
\[L_{C/\ff_q}(T) = \prod_{1 \le j \le 2g}(1 - \alpha_jT)\]
where the roots can be ordered so that
\begin{equation}
\label{eqprodq}
\alpha_j \alpha_{g+j} = q\text{ for each } 1 \le j  \le g.
\end{equation}

Each $\alpha_j$ has archimedean size $\sqrt q$; for each
$\iota:\bar\rat \inject \cx$, one has $\abs{\iota(\alpha_j)} = \sqrt
q$.  In contrast, there is a range of possibilities for the $p$-adic
valuations of the $\alpha_j$.  The Newton polygon of $C$ (or of its
Jacobian $X$) is a combinatorial device which encodes these
valuations.

Let $\kk$ be a field with a discrete valuation $v$, and let $h(T) = \sum
a_iT^i\in \kk[T]$ be a polynomial.  The Newton polygon of $h(T)$ is
defined in the following way.

In the plane, graph the points $(i, v(a_i))$, and form its lower
convex hull.  This object is called the Newton polygon of $h$.
Equivalently, it suffices to track the multiplicity $e(\lambda)$ with
which each slope $\lambda$ occurs in the diagram.  Thus, we will often
record a Newton polygon as the function
 \[
 \xymatrix{
 \rat \ar[r] & \integ_{\ge 0} \\
 \lambda \ar@{|->}[r] & e(\lambda)
 }
 \]
which, to each $\lambda$, assigns the length of the projection of
the ``slope $\lambda$'' part of the Newton polygon onto its first
coordinate.    This function encodes the valuation of the roots of
$h$.  More precisely, it is not hard to check that the number of
$\alpha \in \bar \kk$ such that $v(\alpha) = -\lambda$ and $h(\alpha)
= 0$ is $e(\lambda)$.

Write $\lambda = a_\lambda/b_\lambda$ with $\gcd(a_\lambda,b_\lambda)
= 1$.   Since $h(T)$ is defined over $\kk$, there is an integrality constraint
\begin{equation}
\label{eqintegrality}
e(\lambda) \lambda \in \integ \text{ for each }\lambda \in \rat, 
\end{equation}
which implies that the line segments of the Newton polygon break at points with integral coordinates.
Also, 
\begin{equation}
\label{eqdegree}
\sum_\lambda e(\lambda)/b_\lambda = \deg h.
\end{equation}
We will often
work with the equivalent data
\[
m(\lambda) := e(\lambda)/b_\lambda.
\]

Now equip $\rat$ with the $p$-adic valuation, normalized so that $v(q)
= 1$.  The Newton polygon of $C/\ff_q$ is that of $L_{C/\ff_q}(T)$.
The choice of $p$-adic valuation means that the Newton polygon of $C$
is unchanged by finite extension of the base field.  Moreover, the
relation \eqref{eqprodq} implies that
\begin{align}
\label{eqsupport}
e(\lambda) &=0\text{ if }\lambda\not \in \rat\cap [0,1] \\
\label{eqsymmetry}
e(\lambda) &= e(1-\lambda).
\end{align}
A Newton polygon satisfying \eqref{eqintegrality}, \eqref{eqsupport}
and \eqref{eqsymmetry} will be called an admissible symmetric Newton
polygon of height $\sum_\lambda e(\lambda)/b_\lambda$.  

\subsubsection{Examples}

Let $E/\ff_q$ be an elliptic curve.  There is an integer $a$ such that
$\abs a \le 2 \sqrt q$ such that
\[
\#E(\ff_q) = 1-a+q.
\]
Then 
\[Z_{E/\ff_q}(T) = \frac{1-aT+ qT^2}{(1-T)(1-qT)}.\]
Suppose $\gcd(a,p) = 1$.  (This is the generic case.)  Then the
Newton polygon of $E$ is the lower convex hull of the points 
\[
\st{(0,1),(1,0),(2,1)},
\]
and the slopes of $E$ are $\st{0,1}$; we have
\[
m(\lambda) = \begin{cases}
1 & \lambda \in \st{0,1} \\
0 &\text{otherwise}
\end{cases}.
\]
Such an elliptic curve is
called ordinary. 

Suppose to the contrary that $\gcd(a,p) = p$.  Then the Newton polygon
of $E$ is the lower convex hull of points
\[
\st{(0,1),(1, \ge \frac 12), (2,1)},
\]
and the only slope of $E$ is $\st{1/2}$; $m(1/2) = 1$, and all other
multiplicities are zero. Such an elliptic curve is
called supersingular.   (See Figure \ref{figecnp}.)

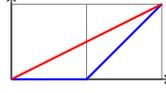
\begin{figure}
\begin{center}
\begin{tikzpicture}
\draw [gray,very thin](0,0) grid (2,1);
\draw[<->,black] (0,1.1) -- (0,0) -- (2.1,0);
\draw [blue,thick](0,0) -- (1,0) -- (2,1);
\draw [red,thick](0,0) -- (2,1);
\end{tikzpicture}
\caption{Newton polygons for elliptic curves.}\label{figecnp}
\end{center}
\end{figure}

We will use the next example (where $g=4$) in the proof of Lemma \ref{lemg4}.

\begin{example} \label{Ep=3g=4}
There exists a (hyperelliptic) curve of genus $4$ defined over $\ff_3$ whose Newton polygon has slopes $1/4$ and $3/4$.
\end{example}

\begin{proof}
Using Sage, we calculated a list of possibilities for monic
polynomials $f(x) \in \ff_3[x]$ of degree nine such that the  
hyperelliptic curve $C:y^2=f(x)$ has $3$-rank $0$.
One such possibility is $$f(x)=x^9 + x^7 + x^6 + 2x^5 + x^4 + 2x^3 + x^2 + x.$$

For the prime $3$, Sage cannot directly compute the L-polynomial of $C$.
Instead, we compute the degree 4 truncation of the zeta function of $C$ to be:
$$127T^4 + 40T^3 + 13T^2 + 4T + 1.$$
The zeta function of $C$ has the form $Z_{C/\ff_3}(T)=L_{C/\ff_3}(T)/(1-T)(1-3T)$
where $$L_{C/\ff_3}(T) = 1+aT +bT^2 +c T^3 +d T^4 +3c T^5 +9bT^6 + 27aT^7 + 81T^8$$
for some coefficients $a,b,c,d$. 
By taking the degree four Taylor polynomial of $Z_{C/\ff_3}(T)$, we solve $a=b=c=0$ and $d=6$.
Then the slopes of the Newton polygon of $L_{C/\ff_3}(T) = 81T^8 + 6T^4 + 1$ are $1/4$ and $3/4$.
\end{proof}

\subsubsection{Newton polygon of an abelian variety over an arbitrary
  field}
\label{subsubsecnparbitrary}

There is an equivalent notion of Newton polygon which also makes sense
for an arbitrary abelian variety over an arbitrary field of positive
characteristic.  For each $\lambda \in \rat_{\ge 0}$, write $\lambda =
a_\lambda/b_\lambda = a_\lambda/(a_\lambda+c_\lambda)$ with $a_\lambda$ and $b_\lambda$ relatively
prime.  Manin defines a certain $p$-divisible group $G_\lambda =
G_{a_\lambda, c_\lambda}$ over $\ff_p$, with dimension $a_\lambda$,
codimension $c_\lambda$ and, thus, height $b_\lambda$ \cite{manincfg}.
If $G$ is any $p$-divisible group over an algebraically
closed field $k$, then there is an isogeny $G \ra \oplus_\lambda
H_\lambda^{\oplus m_G(\lambda)}$.  The isogeny itself is not
canonical, but the collection of all nonnegative integers
$\st{m_G(\lambda)}$ is independent of all choices.    Let
$e_G(\lambda) = b_\lambda m_G(\lambda)$.

Now let $X/K$ be an abelian variety of dimension $g$; the Newton
polygon of $X$ is that of its $p$-divisible group $X[p^\infty]_{\bar
  K}$.  It is not hard to verify that $e_{X[p^\infty]}$ is an
admissible symmetric Newton polygon of height $2g$.  In the special
case where $C/\ff_q$ is a smooth projective curve over a finite field,
the Newton polygon of $C$ (as defined in Section
\ref{subsubnpcurve}) coincides with that of its Jacobian.  In
general, we define the Newton polygon of $C/K$ as that of its
Jacobian.

The Newton polygon of a $p$-divisible group, and thus of an abelian
variety, is invariant under isogeny.   (Note, however, that even if
$X[p^\infty]$ and $Y[p^\infty]$ are isogenous, it does {\em not} follow that $X$ and
$Y$ are isogenous.)  Moreover, if $G$ and $H$ are
$p$-divisible groups, it follows immediately from the definition that
\[
m_{G\oplus H}(\lambda)  = m_G(\lambda) + m_H(\lambda).
\]

The symmetry condition $m(\lambda) = m(1-\lambda)$ forces
the inequality $0 \le f_X \le \dim X$ noted in Section \ref{subsecprank}.

An abelian variety $X/K$ is ordinary if and only if all its slopes are $0$ and $1$.
The $p$-rank of $X$ is equal
to the multiplicity $m(0)$ of the slope $0$ in the Newton polygon. 
An abelian variety is {\it supersingular} if all the slopes of its Newton polygon equal $1/2$.
Thus, if an abelian variety is supersingular, then it has $p$-rank zero.  However, in
dimension at least three, the converse is false.  
For example, there are abelian varieties with $p$-rank $0$ whose Newton polygons have slopes $\frac{1}{g}$ and $\frac{g-1}{g}$
and are thus not supersingular for $g \geq 3$.

\subsection{Semicontinuity and purity}

We now consider a family of $p$-divisible groups, such as that coming
from a family of abelian varieties in characteristic $p$.  It is not too hard to show that
the $p$-rank is a lower semicontinuous function, i.e., that it can
only decrease under specialization.  In fact, if the $p$-rank does
change, it does so in codimension one:

\begin{lemma}
\label{lempurityprank}
 \cite[Lemma 1.6]{oort74}
Let $X \ra S$ be an abelian variety over an integral scheme in
positive characteristic, and suppose that $X$ has generic $p$-rank
$f$.  Let $S^{<f}\subset S$ denote the locus parametrizing those $s$
such that $X_s$ has $p$-rank strictly less than $f$.  Then either
$S^{<f}$ is empty or $S^{<f}$ is pure of codimension one.
\end{lemma}

There are analogous semicontinuity and purity results for Newton
polygons, although the former requires more notation to state, and the
proof of the latter is much deeper than that of Lemma \ref{lempurityprank}.

The partial ordering on Newton polygons is defined as follows. Let
\[\nu = \st{\lambda_1^{\oplus m_{\lambda_1}}, \ldots, \lambda_r^{\oplus
m_{\lambda_r}}},\]
where $\lambda_i < \lambda_{i+1}$.  Let
$\Gamma(\nu) \subset {\mathbb R}^2$ be the convex hull of $(0,0)$ and
\[\st{(\sum_{i=1}^j m_{\lambda_i}c_{\lambda_i}, \sum_{i=1}^j
  m_{\lambda_i}a_{\lambda_i}): 1 \le j \le r}.\]  
If $\mu$ and $\nu$ are two Newton polygons, we will
write $\mu \preceq \nu$ if $\Gamma(\mu)$ and $\Gamma(\nu)$ have the
same endpoints and if all points of $\Gamma(\mu)$ lie on or above
those of $\Gamma(\nu)$.  (This convention may seem a little
surprising, but it has the pleasant consequence that ``smaller''
Newton strata are in the closures of ``larger'' ones.)  See Figure
\ref{figposetg4} for the $g=4$ case of  the poset of admissible symmetric
Newton polygons.

Newton polygons have the following semicontinuity property \cite[Theorem
2.3.1]{katzsf}: Let $S = \spec(R)$ be the spectrum of a local ring,
with geometric generic point $\bar\eta$ and geometric closed point
$\bar s$.  If $G$ is a $p$-divisible group over $S$, then $\nu(G_{\bar
  s}) \preceq \nu(G_{\bar \eta})$.  Like the $p$-rank, the Newton
polygon is a discrete invariant which changes (if at all) in
codimension one.

\begin{proposition}
\label{propdejongoort}
\cite{dejongoort}
Let $S$ be an integral, excellent scheme.  Let $G \ra S$ be a
$p$-divisible group over $S$, and let $U\subset S$ be the largest
dense set on which the Newton polygon is constant.  Then either $U =
S$ or $\codim((S\setcomp U), S) = 1$.
\end{proposition}

\subsection{Notation on stratifications and Newton polygons}

Let $X \ra S$ be any family of abelian varieties of relative dimension
$g$ in positive
characteristic.  For any Newton polygon $\nu$, let $S^\nu$ be the
reduced subspace such that $s \in S^\nu$ if and only if the Newton
polygon of $X_s$ is $\nu$.  In general, the Newton stratification
refines the $p$-rank stratification $S = \sqcup S^f$, where $s\in S^f$
if and only if the $p$-rank of $X_s$ is $f$.  

If $C \ra S$ is a family of smooth, projective curves, then the Newton polygon
and $p$-rank strata of $S$ are those corresponding to the relative
Jacobian $\jac(C) \ra S$.

Suppose that $S$ is irreducible and $X \ra S$ is a family of abelian
varieties.  Since there are only finitely many symmetric admissible Newton polygons of any
particular height, by semicontinuity there is a nonempty open
subset $U$ over which the Newton polygon of $X$ is constant.   We call
the Newton polygon of (any geometric fiber of) $X_U \ra U$ the generic
Newton polygon of $X \ra S$, or simply of $S$ if the family of abelian
varieties is clear from context.  

\begin{definition}
\label{defnugf}
For each $0 \le f \le g$, define a Newton polygon $\nu_g^f$ as follows:
\begin{align*}
\nu_g^g &= \st{ 0^{\oplus g}, 1^{\oplus g}} \\
\nu_g^0 &= \begin{cases}
\st{\half 1 ^{\oplus g}} & g \le 2 \\
\st{\oneover g, \frac{g-1}{g}} & g \ge 3
	     \end{cases}
\\
\intertext{and if $0 < f < g$, set}
\nu_g^f &= \nu_f^f \oplus \nu_{g-f}^0 .
\end{align*}
This is the largest (``most generic'') admissible symmetric Newton polygon of height $2g$
and $p$-rank $f$.  Similarly, let
\[
\sigma_g = \st{\half 1 ^{\oplus_g}}
\]
be the supersingular Newton polygon of height $2g$.
\end{definition}

With this
notation, we depict the poset of Newton polygons for $g=4$ in Figure \ref{figposetg4}.

\begin{figure}
\begin{center}
\begin{tikzpicture}[xscale=2]
\node (sigma4) at (0,0) {$\sigma_4$};
\node (nu30sigma1) at (1,0) {$\nu_3^0\oplus \sigma_1$};
\node (nu40) at (2,-1) {$\nu_4^0$};
\node (nu11sigma3) at (2,1) {$\nu_1^1\oplus \sigma_3$};
\node (nu11nu30) at (3,0) {$\nu_1^1\oplus \nu_3^0$};
\node (nu42) at (4,0) {$\nu_4^2$};
\node (nu43) at  (5,0) {$\nu_4^3$};
\node (nu44) at  (6,0) {$\nu_4^4$};
\draw [->] (sigma4) -- (nu30sigma1);
\draw [->] (nu30sigma1) -- (nu11sigma3);
\draw [->] (nu30sigma1) -- (nu40);
\draw [->] (nu11sigma3) -- (nu11nu30);
\draw [->] (nu40) -- (nu11nu30);
\draw [->] (nu11nu30) -- (nu42);
\draw [->] (nu42) -- (nu43);
\draw [->] (nu43) -- (nu44);
\end{tikzpicture}
\caption{ Symmetric admissible Newton polygons of height $8$.  There is
a path from $\nu$ to $\nu'$ if and only if $\nu \prec \nu'$.}\label{figposetg4}
\end{center}
\end{figure}
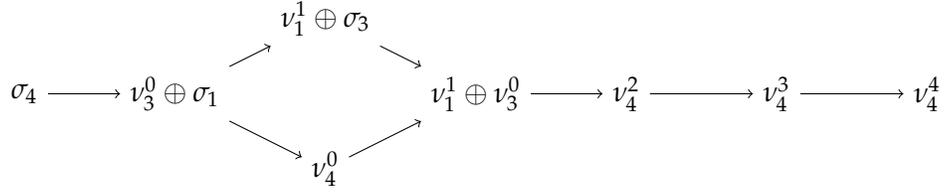

For a symmetric admissible Newton polygon $\nu$ of height $2g$, let
$\codim(\nu)$ be the length of any path from $\nu_g^g$ to $\nu$ in
the poset of all such Newton polygons.

\section{Stratifications on the moduli space of abelian varieties}
\label{secag}

The $p$-rank and Newton stratifications of $\stk A_g$ are  well
understood.   We review some of their features here, for two different reasons.

First, for $1 \le g \le 3$, $\stk M_g \inject \stk A_g$ is an open immersion, and $\stk
M_g^* \ra \stk A_g$ is an isomorphism.  Consequently, in small genus,
the stratifications on $\stk A_g$ are perfectly mirrored in $\stk
M_g$.   This information is used directly below.

Second, as we shall see, $\stk A_g$ is in some sense more highly structured than
$\stk M_g$.  Thus, the results described here for $\stk A_g$ might
elicit some of the most optimistic conjectures one might make about such stratifications
on $\stk M_g$.  (Indeed, some ``obvious'' conjectural statements are
too strong to be true; see Section \ref{subsecnonexistence}.)

\subsection{The $p$-ranks of abelian varieties}

\begin{theorem}
\cite{normanoort}
\label{thnormanoort}
Let $g \ge 1$ and $0 \le f \le g$.  Then $\stk A_g^f$ is nonempty
and pure of codimension $g-f$ in $\stk A_g$.
\end{theorem}

The nonemptiness is easy to see; it suffices to take a product of $f$
ordinary elliptic curves and $g-f$ supersingular elliptic curves,
equipped with the product principal polarization.

\begin{corollary}  Let $g \ge 1$ and $0 \le f \le g$.  Let $S$ be an
  irreducible component of $\stk A_g^f$.
\begin{alphabetize}
\item If $f< g$, then $S$ is in the closure of $\stk A_g^{f+1}$.
\item If $f> 0$, then the closure of $S$ contains an irreducible
  component of $\stk A_g^{f-1}$.
\end{alphabetize}
\end{corollary}

\begin{proof}
Part (a) follows from the dimension count of Theorem
\ref{thnormanoort} and purity for $p$-ranks (Lemma
\ref{lempurityprank}).

We do not know an elementary proof of part (b), although it is an
immediate consequence of Theorem \ref{thchaioort}.
\end{proof}

Note that, unless $(g,f) = (1,0)$ or $(2,0)$, Theorem \ref{thchaioort}
implies that $\stk A_g^f$ is actually irreducible.

\subsection{Newton polygons of abelian varieties}

Thanks especially to work of Oort and of Chai and Oort, the Newton
stratification on $\stk A_g$ is well-understood.

\begin{theorem}
\cite{chaioortannals11}
\label{thchaioort}
Let $\nu$ be a symmetric admissible Newton polygon of height $2g$.
\begin{alphabetize}
\item The stratum $\stk A_g^\nu$ is nonempty and has codimension
  $\codim(\nu)$ in $\stk A_g$.
\item If $\nu \not = \sigma_g$, then $\stk A_g^\nu$ is irreducible.
\item The supersingular locus $\stk A_g^{\sigma_g}$ is connected but
  reducible.
\end{alphabetize}
\end{theorem}

In analyzing Newton strata $\stk A_g^\nu$, two key tools are the
Serre-Tate theorem and the action of Hecke operators.  In concert
with Dieudonn\'e theory, the first tool allows one to use
(semi-)linear algebra to study the deformation space of an abelian
variety.  The second tool, and in particular the fact that Newton strata
are stable under this large group of partial symmetries of $\stk A_g$,
allows one to deduce global information about Newton strata.

In general, both of these structures are missing from $\stk M_g$, and
the state of our knowledge is, correspondingly, much cruder.  For a given symmetric admissible Newton polygon $\nu$ of
height $2g$, it is typically not even known if $\stk M_g^\nu$ is nonempty, let
alone pure or even irreducible.  {\em Degeneration} (as in Theorem
\ref{propmgboundary}) and {\em low genus
  phenomena} (e.g., the fact that every principally polarized abelian
threefold is the Jacobian of a stable curve) are among the few tools
we have at our disposal.  In the second half of this paper, we show how these
techniques can be combined to yield information about stratifications
on $\stk M_g$.

\section{The $p$-rank stratification of the moduli space of stable curves}
\label{secmgbar}

Let $\stk M_g/\ff_p$ be the  moduli stack of
smooth proper curves of genus $g$.  Even if one is intrinsically only
interested in smooth curves, one is quickly led, following Deligne and
Mumford, to study $\barstk M_g$, the moduli stack of stable proper curves
of genus $g$ \cite{delignemumford}.  

\subsection{The moduli space of stable curves}
\label{subsecdm}

It turns out that $\stk M_g$ is
open in $\barstk M_g$, which is proper.  The boundary $\del\barstk M_g = \barstk M_g \setcomp \stk M_g$ is
a union $\del\barstk M_g = \cup_{0 \le i \le \floor{g/2}}\Delta_i[\barstk M_g]$,
whose construction we briefly recall.

For a natural number $r$, let $\barstk M_{g;r}$ be the moduli stack of
$r$-labeled stable curves of genus $g$.  There are finite clutching
morphisms 
\[
\xymatrix{
\bar\calm_{g_1;1} \cross \bar\calm_{g_2;1} \ar[r]^{\kappa_{g_1,g_2}} &
\bar\calm_{g_1+g_2}\\
\bar\calm_{g;2} \ar[r]^{\kappa_{g}} & \bar\calm_{g+1}
}
\]
in which the labeled points are identified; see \cite[Section
3]{knudsen2} for more details.   Here, we simply record the following
facts.

Suppose that, for $i \in \st{1,2}$, $(C,P_i)$ is a smooth, pointed
curve with moduli point $s_i  \in \stk M_{g_i;1}(k)$. Then
$\kappa_{g_1,g_2}(s_1,s_2)$ is the moduli point of the curve $C$ of genus
$g_1+g_2$ obtained by identifying $P_1$ and $P_2$.  One has
\begin{equation}
\label{eqpicdeltai}
\pic^0(C) \iso \pic^0(C_1) \oplus \pic^0(C_2).
\end{equation}

Now suppose that $(C,P,Q) \in \stk M_{g;2}(k)$ is a smooth,
$2$-pointed curve.  Then $\kappa_{g}(s)$ is the moduli point of the
curve $\til C$ obtained from $C$ by identfying $P$ and $Q$, and there
is an exact sequence
\begin{equation}
\label{diagpicdelta0}
 \xymatrix{
 1 \ar[r] & \gp_m \ar[r] & \pic^0(\til C) \ar[r] & \pic^0(C) \ar[r] & 1.
 }
\end{equation}

Let $\Delta_0[\barstk M_g] =
\kappa_{g-1}(\barstk M_{g-1;2})$, and for  $1 \le i \le g-1$
let $\Delta_i[\barstk M_g] = \kappa_{i,g-i}(\barstk M_{i;1} \cross
\barstk M_{g-i;1})$.   If $S$ is a stack equipped with a morphism $S
\ra \barstk M_g$, we let $\Delta_i[S] = S\cross_{\barstk M_g}
\Delta_i[\barstk M_g]$.  Let $\stk M_g^* = \barstk M_g \setcomp \Delta_0$; this
is precisely the locus of curves whose Picard varieties are actually
abelian varieties, and not merely semiabelian varieties. 
The Torelli
map $\tau: \stk M^*_g \ra \stk A_g$ is a birational morphism onto its
image, and contracts fibers on the boundary.  More precisely, (the Torelli theorem states that) $\tau$ is injective on $\stk M_g$.  On the boundary, $\tau$ forgets the identification point; if $P_1$ and $Q_1$ are points on a curve $C_1$, then $\tau((C_1,P_1),(C_2,P_2)) = \tau((C_1,Q_1),(C_2,P_2))$.

\subsection{The $p$-rank stratification of $\barstk M_g$}
\label{subsecmgbarf}

The notion of $p$-rank makes sense for stable, and not just smooth,
curves. 

On $\Delta_0$ we find, in the notation of \eqref{diagpicdelta0}, that
\begin{align}
\label{eqprankdelta0}
f_{\til C} &= f_{C}+1;
\intertext{while on $\Delta_i$ with $i>0$ we find, in the notation of
\eqref{eqpicdeltai}, that}
\label{eqprankdeltai}
f_C &= f_{C_1} + f_{C_2}.
\end{align}
Thus, the $p$-rank stratification extends to $\barstk M_g$.  We
emphasize that $\barstk M_g^f$ parametrizes those stable curves whose
$p$-rank is $f$, while $\bar{\stk M_g^f}$ is the closure of the set of
smooth curves of genus $g$ and $p$-rank $f$.  In particular, if $f>0$,
one has $\barstk M_g^f \subsetneq \bar{\stk M_g^f}$.

From \eqref{eqprankdelta0}
and \eqref{eqprankdeltai} it immediately follows that
\begin{align*}
\kappa_{g_1,g_2}(\barstk M_{g_1}^{f_1}\cross \barstk M_{g_2}^{f_2})
&\subseteq \barstk M_{g_1+g_2}^{f_1+f_2}\\
\kappa_{g-1}(\barstk M_{g-1}^f) & \subseteq \barstk M_g^{f+1}.
\end{align*}
Faber and van der Geer exploit this structure to show:
\begin{theorem}
\cite{fvdg}
\label{thdimmgf}
Suppose $g \ge 1$ and $0 \le f \le g$.  Then $\stk M_g^f$ is nonempty
and pure of codimension $g-f$ in $\stk M_g$.
\end{theorem}
In fact, much more is true.  First, using relations \eqref{eqprankdelta0}-\eqref{eqprankdeltai} and the fundamental dimension count supplied by Theorem \ref{thdimmgf}, it
is not hard to see that $\stk M_g^f$ is dense in $\barstk M_g^f$;
every stable curve of genus $g$ and $p$-rank $f$ is a limit of smooth
curves with the same discrete parameters.  Moreover, the recursive
structure of the boundary is compatible with the $p$-rank stratification:

\begin{theorem}
\label{propmgboundary}
\cite[Lemma 3.2 and Prop.\ 3.4]{achterpriesprank}
Suppose $g \ge 2$ and $0 \le f \le g$.  Let $S$ be an irreducible
component of $\stk M_g^f$.  
\begin{alphabetize}
\item If $f > 0$, then $\bar S$ contains the image of an irreducible
  component of $\barstk M_{g-1;2}^{f-1}$ under $\kappa_{g-1}$.
\item Suppose $1 \le i \le g-1$.  Let $f_1$ and
$f_2$ be nonnegative integers such that $0 \le f_1 \le i$; $0 \le f_2
\le g-i$; and $f_1+f_2 = f$.  Then $\bar S$ contains the image of an
irreducible component of $\barstk M_i^{f_1} \cross
\barstk M_{g-i}^{f_2}$ under $\kappa_{i,g-i} $.
\end{alphabetize}
\end{theorem}

Consequently, closures of components of $p$-rank strata contain
chains of elliptic curves:

\begin{corollary}
\label{cortree}
\cite[Corollary 3.6]{achterpriesprank}
Suppose $g \ge 2$, $0 \le f \le g$, and $A\subset\st{1, \cdots, g}$
has cardinality $f$.  Let $S$ be an irreducible component of $\stk
M_g^f$.  Then $\bar S$ contains the moduli point of a chain of
elliptic curves $E_1, \cdots, E_g$, where $E_j$ is ordinary if and
only if $j \in A$.
\end{corollary}

\subsection{Connectedness of $p$-rank strata} \label{Sprank}
\label{secstratprank}

We combine Theorem \ref{thdimmgf} with degeneration techniques to prove:
\begin{corollary}  \label{PnestedMg}
Suppose $g \ge 1$ and $0 \le f \le g$.  Let $S$
  be an irreducible component of $\stk M_g^f$.
\begin{alphabetize}
\item If $f<g$, then $S$ is in the closure of $\stk M_g^{f+1}$.
\item If $f >0$, then $\bar S$ contains an irreducible component of
  $\stk M_g^{f-1}$.
\end{alphabetize}
\end{corollary}

\begin{proof}
Part (a) is a direct consequence of purity for $p$-rank (Lemma
\ref{lempurityprank}) and the dimension count Theorem
\ref{thdimmgf}; the proof proceeds by induction on $g-f$.

For part (b), the statement is clearly true for $g=1$, or more
generally when $f=g$.  Now suppose $g \ge 2$.  By Theorem
\ref{propmgboundary}(b), $\bar S$ contains an irreducible component of
$\barstk M_1^1 \cross \barstk M_{g-1}^{f-1}$ under the image of
$\kappa_{1,g-1}$.  Since $\stk M_1^1$ is irreducible, its closure
contains points of $p$-rank zero, and $\Delta_1[\bar S]^{f-1}$ is
nonempty.  Therefore, $\bar S^{<f}$ is nonempty and, by Lemma
\ref{lempurityprank}, has codimension one in $\bar S$.  For each $i$
between $0$ and $g$, $\Delta_i[\bar S]^{<f}$ has codimension two in
$\bar S$.  Therefore, $\bar S^{<f}$ is the closure of $S^{<f}$, and the
latter has codimension one in $S$.  The basic dimension count
(Theorem \ref{thdimmgf}) now shows that $\bar S$ contains an
irreducible component of $\stk M_g^{f-1}$.
\end{proof}

In contrast with the theory of $p$-rank strata of $\stk A_g$, at
present we have essentially no nontrivial information about
irreducibility of various $\stk M_g^f$.  Still, the method of
degeneration detects a small amount of the topology of these strata.
Recall that $\barstk M_g^f$ denotes the $p$-rank $f$ stratum of the moduli space of 
stable curves of genus $g$.

\begin{corollary}
\label{thmgconn}
If $g \ge 2$ and $0 \le f \le g$, then $\barstk M_g^f$ is connected.
\end{corollary}

\begin{proof}
Since $p$-rank strata in $\stk M_3^*$ coincide with those of
$\stk A_3$, and since $p$-rank strata in $\stk A_g$ are irreducible
(Theorem \ref{thchaioort}), each $\stk M_3^f$ is irreducible.  
Similarly, $\stk M_2^1$ and $\stk M_2^2$ are irreducible, while
$\barstk M_2^0$ is connected.

We next prove by induction on $g$ that, for each $g \ge 4$ and each $0
\le f \le g$, $\barstk M_g^f$ is connected.  Fix integers $f_1$ and
$f_2$ with $0 \le f_1 \le 2$, $0 \le f_2 \le g-2$, and $f_1+f_2 = f$.
Let $S_1$ and $S_2$ be irreducible components of $\stk M_g^f$.  By
Theorem \ref{propmgboundary}(b), each closure $\bar S_i$ contains
an irreducible component of $\kappa_{2,g-2}(\barstk M_2^{f_1} \cross
\barstk M_{g-2}^{f-2})$.  By the inductive hypothesis, $\barstk M_2^{f_1} \cross
\barstk M_{g-2}^{f-2}$ is connected.  Consequently, $\bar{S_1 \cup
  S_2}$ is connected.  The theorem now follows by considering each of
the (finitely many) irreducible components of $\stk M_g^f$.
\end{proof}

\subsection{Open questions about the $p$-rank stratification}
\label{subsecopenprank}

The results above indicate that the $p$-rank stratification is in some ways ``transverse'' to
other interesting loci in $\stk M_g$.  For example, as long as $(g,f)
\not = (1,0)$ or $(2,0)$, then there exists a smooth
projective curve over $\bar\ff_p$ of genus $g$ and $p$-rank $f$ whose
automorphism group is trivial \cite{achterglasspries}.  More
generally, there is a precise sense in which a randomly chosen
Jacobian of genus $g$ and $p$-rank $f$ behaves like a randomly
selected principally polarized abelian variety of dimension $g$
\cite{achterpriesprank}.

On the other hand, there is a non-trivial interplay between the $p$-rank and the automorphism group of a curve.
For example, there are constraints on the $p$-rank of a cyclic tame cover of ${\mathbb P}^1$ \cite{bouw01}.
There are even stronger constraints in the case of a wildly ramified Galois action. 
The Deuring-Shafarevich formula places severe limitations on the $p$-rank for a $p$-group cover of curves \cite{crew84,subrao75}.
  The $p$-rank stratification of the moduli space of Artin-Schreier curves is discovered in \cite{prieszhu}.  See also \cite{GKaut, glass2}.

Here are several other open questions about the $p$-ranks of Jacobians of curves.

\begin{question} \label{Qp1}
Suppose $g \geq 4$ and $0 \leq f \leq g-1$.  Is $\calm_g^f$ irreducible?
\end{question}

\begin{question}
Suppose $g \geq 3$ and $0 \leq f \leq g-1$.  Does there exist a curve of genus $g$ and $p$-rank $f$ defined over $\ff_p$?
\end{question}

\section{Stratification by Newton polygon}
\label{secstratnp}

In this section, we explain how Theorem \ref{propmgboundary}
yields information about generic Newton polygons, beginning with the cases $g \leq 4$, and extending to 
arbitrary $g$.  Recall (Definition \ref{defnugf}) the Newton polygons $\nu_g^f$.

\subsection{Newton polygons of curves of small genus}

\begin{lemma}
\label{lemg3}
The $p$-rank zero locus $\stk M_3^0$ is irreducible, with generic
Newton polygon $\nu_3^0= \st{1/3, 2/3}$.
\end{lemma}

\begin{proof}
The only symmetric Newton polygons of height $6$ and $p$-rank zero are
$\nu_3^0$ and $\sigma_3 = \st{\half 1^{\oplus 3}}$; but the supersingular locus
in $\stk A_3^0$ is pure of dimension $2$ (\ref{thchaioort}), while
$\stk M_3^0$ is pure of dimension $3$ (\ref{thdimmgf}).  The result
now follows from the Torelli theorem and  Theorem
\ref{thchaioort}.
\end{proof}

\begin{lemma}
\label{lemnotallss}
Let $S$ be an irreducible component of $\stk M_4^0$.  Then the generic
Newton polygon of $S$ lies on or below $\sigma_1\oplus \nu_3^0= \st{1/3, 1/2, 2/3}$.
\end{lemma}

\begin{proof}
By Theorem \ref{propmgboundary}, $\bar S$ contains a component of
$\kappa_{1,3}(\barstk M_{1;1}^0 \cross \barstk M_{3;1}^0)$, which has
generic Newton polygon $\st{1/2} \oplus \st{1/3,2/3}$ (Lemma
\ref{lemg3}).   The result for
$S$ follows from the semicontinuity of Newton polygons.
\end{proof}

\begin{lemma}
\label{lemg4}
In $\stk M_4^0$, 
\begin{alphabetize}
\item\label{partatleast} there is at least one irreducible component with
  generic Newton polygon $\nu_4^0$; and 

\item\label{partatmost} there is at most one irreducible component with
  generic Newton polygon $\nu_3^0\oplus \sigma_1$.
\end{alphabetize}
\end{lemma}

\begin{proof}
  Let $S$ be an irreducible component of $\stk M_4^0$; then $S$ has
  dimension $5$ (Theorem \ref{thdimmgf}).  Suppose the generic Newton polygon of $S$ is
  $\nu_3^0\oplus \sigma_1$.  The locus in $\stk A_4$ of abelian
  fourfolds with Newton polygon $\nu_3^0\oplus \sigma_1$ is
  irreducible of dimension $5$ (Theorem \ref{thchaioort}).  Thus $S$,
  or rather its image
  $\tau(S)$ under the Torelli map, must coincide with this locus.  In
  particular, such an $S$, if it exists, is unique; this proves
  \eqref{partatmost}.

For part \eqref{partatleast}, it suffices to show that there exists
some curve whose Jacobian has Newton polygon $\nu_4^0$.  Consider
the moduli space $\stk Z$ of principally polarized abelian fourfolds
equipped with an action by $\integ[\zeta_3]$ of signature $(3,1)$.
Then $\stk Z$ is contained in the (compactified) Torelli locus.  

This has been known for some time,  but we
provide a sketch here. 
For a squarefree polynomial $f(x)$ of degree
$6$, let $C_f$ denote the curve with affine equation $y^3 = f(x)$.
Then $\jac(C_f)$ has an action by $\integ[\zeta_3]$ of signature
$(3,1)$.  On one hand, the parameter space for such curves has
dimension $6 - \dim {\rm PGL}_2 = 3$.  On the other hand, $\stk Z$ itself is
irreducible (since $\rat(\zeta_3)$ has class number one) of dimension
$3 \cdot 1 = 3$.  Consequently, $\tau(\stk M_4)$ contains an open,
dense subspace of $U$, and $\tau(\stk M_4^*)$, which is equal to the
closure of $\stk M_4$ in $\stk A_4$, contains all of $\stk Z$.

It now suffices to show that there exists a point on $\stk Z$
parametrizing an abelian variety with Newton polygon $\nu_4^0$.  If
$p$ splits in $\rat(\zeta_3)$, this follows from \cite[Section
2.2]{mantovanunitary}.  (In the notation of \cite{mantovanunitary},
$(q,h-q) = (1,3)$, and $\alpha$ is the Newton polygon of slope $1/4$.)
If $p$ is inert in $\rat(\zeta_3)$, this follows from
\cite[Section 5.4]{bultelwedhorn}.  (In the notation of \cite{bultelwedhorn},
$(n-s,s) = (3,1)$, $\rho = 4$.)  
For $p=3$, this follows from Example \ref{Ep=3g=4}.
\end{proof}

\subsection{Generic Newton polygons}

\begin{proposition} \label{Pinduction}
Let $g_0, f \in {\mathbb N}$ and let $g=g_0+f$.
\begin{alphabetize}
\item 
If the generic Newton polygon of every component of $\calm_{g_0}^0$ is $\nu_{g_0}^0$, 
then the generic Newton polygon of every component of $\calm_g^f$ is $\nu_{g}^f$.

\item If the generic Newton polygon of at least one component of $\calm_{g_0}^0$ is $\nu_{g_0}^0$
then the generic Newton polygon of at least one component of $\calm_{g}^f$ is $\nu_g^f$.
\end{alphabetize}
\end{proposition}

\begin{proof}
For (a), let $S$ be a component of $\stk M_g^{f}$ and consider its closure in $\overline{\calm}_g^f$.
By Theorem \ref{propmgboundary}(b), $\bar S$ contains the image of an irreducible component of $\stk M_{f;1}^{f} \cross \stk M_{g_0;1}^0$.  
Now $\stk M_{f;1}^f$ is irreducible with generic Newton polygon $\nu_g^g$.
By hypothesis, every component of $\stk M_{g_0;1}^0$ has generic Newton polygon $\nu_{g_0}^0$.
By semicontinuity, the generic Newton polygon $\nu$ of $S$ lies on or below $\nu_{g}^f$.
On the other hand, the $p$-rank $f$ condition implies that $\nu$ has $f$ slopes of $0$ and $1$.
This is only possible if $\nu=\nu_g^f$.

The proof of (b) is by induction on $f$.  The base case is the
hypothesis that there exists a component $S_{g_0}$ of $\stk M_{g_0}^0$
with generic Newton polygon $\nu_{g_0}$.  Now suppose, as inductive
hypothesis, that $S_{g-1}$ is a component of $\stk M_{g-1}^{f-1}$which
has generic Newton polygon $\nu_{g-1}^{f-1}$.  We add a labeling of
two points to each curve represented by a point of $S_{g-1}$ by
letting $T_{g-1} = S_{g-1} \cross_{\stk M_{g-1}} \stk
M_{g-1;2}^{f-1}$.  Then consider the image under the clutching
morphism $Z_g = \kappa_{g-1}(T_{g-1})$ which is contained in
$\Delta_0[\barstk M_g^{f}]$.  By \cite[Lemma 3.2]{achterpriesprank},
there exists an irreducible component $S_g$ of $\stk M_g^{f}$ such
that $\bar S_g$ contains $Z_g$.  By semicontinuity, the generic Newton
polygon $\nu$ of $S_g$ lies on or below $\nu_{g}^f$.  Then
$\nu=\nu_g^f$ by the $p$-rank $f$ condition.
\end{proof}

If $g \le 2$, then the $p$-rank of an abelian variety determines its
Newton polygon.  More generally, if $f \in \st{g-2,g-1,g}$, then there
is a unique symmetric admissible Newton polygon of height $2g$ and
$p$-rank $f$.  In contrast, if $f< g-2$ then the $p$-rank
constrains, but does not determine, the Newton polygon.

\begin{corollary}
\label{cormgnpg3}
Let $g \geq 3$.
If $S$ is a component of $\stk M_g^{g-3}$, then $S$ has generic Newton polygon $\nu_{g}^{g-3}=\st{0^{\oplus g-3}, 1/3, 2/3, 1^{\oplus g-3}}$.
\end{corollary}

\begin{proof}
By Lemma \ref{lemg3}, the generic Newton polygon of every component of $\calm_{3}^0$ is $\nu_{3}^0$.
The result follows from Proposition \ref{Pinduction}(a).
\end{proof}

We obtain partial information in the next case when $f=g-4$.

\begin{corollary}
\label{thmgnpg4}
Let $g \ge 4$.  There exists a component of $\stk M_g^{g-4}$ with generic Newton polygon $\nu_{g}^{g-4}$.
In particular, there is a smooth projective curve whose Jacobian has Newton polygon $\nu_{g}^{g-4}$.
\end{corollary}

\begin{proof}
By Lemma \ref{lemg4}\eqref{partatleast}, there exists a component $S_4$ of $\stk M_4^0$ with generic Newton polygon $\nu_{4}^0$.
The result follows from Proposition \ref{Pinduction}(b).
\end{proof}

A better understanding of $\calm_5^0$ would allow one to prove results for arbitrary $g$ when $f=g-5$.

\section{Hyperelliptic curves}  
Recall that a hyperelliptic curve is a smooth projective curve $C$
which  can be realized as a double cover $C \ra {\mathbb P}^1$ of the
projective line. Among all curves, hyperelliptic curves have enjoyed special
attention.  On the practical side, algorithmic methods for handling
hyperelliptic curves over finite fields are much more highly developed
than they are for arbitrary curves.  On the theoretical side,
hyperelliptic curves over finite fields are a natural function-field
analogue of quadratic number fields, and thus an attractive site for
investigation of conjectures.  In this section, we briefly
sketch the extent to which the ideas and results surveyed here for
$\stk M_g$ extend to the moduli space of hyperelliptic curves.
Throughout, we assume that the characteristic $p$ is {\em odd}.

\paragraph{The boundary of $\stk H_g$}

Let $\stk H_g$ be the moduli space of hyperelliptic curves.  Any given curve
admits at most one hyperelliptic involution $\iota$, and thus there is an
inclusion $\stk H_g \inject \stk M_g$.  Let $\barstk H_g$ be the
closure of $\stk H_g$ in $\stk M_g$.     The boundary $\del H_g$
necessarily admits a decomposition $\del H_g = \cup \Delta_i [\barstk H_g]$,
but constructing the full boundary is somewhat delicate.
Briefly, if two hyperelliptic curves $(C_1,P_1)$ and $(C_2,P_2)$, with
hyperelliptic involutions $\iota_1$ and $\iota_2$, are clutched, then
the resulting curve is hyperelliptic if and only if each $P_i$ is fixed by
$\iota_i$.  Consequently, for $0<i<g$, $\Delta_i[\barstk H_g]$ is
described as the image of
\[
\xymatrix{
\tilstk  H_i \cross \tilstk H_{g-i} \ar[r] & \tilstk H_g \ar[r] & \barstk H_g
}
\]
where $\tilstk H_g$ is the moduli space of hyperelliptic curves equipped
with a labeling of the $2g+2$ points of the smooth ramification locus,
and $\tilstk H_g \ra  \barstk H_g$ is the (finite-to-one) forgetful
map.  

The description of irreducible components of $\Delta_0[\barstk H_g]$
is somewhat more intricate, since the curve obtained by identifying
points $P$ and $Q$ of the hyperelliptic curve $C$ is hyperelliptic if
and only if $\iota(P) = Q$.  In fact, there is a decomposition of 
$\Delta_0[\barstk H_g]$ as a union 
\[
\Delta_0[\barstk H_g]  = \kappa_{g-1}(\barstk H_{g-1};1) \bigcup_{1
  \le i \le \floor{g½}} \lambda_{i,g-1-i}(\barstk H_{i;1} \cross \barstk H
_{g-1-i;1})  
\]
of irreducible divisors.  We refer to \cite{achterprieshyper,yamaki04} for more details.

The fact that, for $i>0$, $\Delta_i[\barstk H_g]$ is parametrized by
hyperelliptic curves with a {\em discrete} choice of point makes it
much more difficult to unravel the $p$-rank stratification $\del\stk
H_g$.  

As in Theorem \ref{thdimmgf}, it turns out that if $0 \le f \le
g$, then $\stk H_g^f$ is nonempty and pure of codimension $g-f$
\cite{glasspries}.  

This is used in proving the currently optimal hyperelliptic analogue to Theorem
\ref{propmgboundary}:

\begin{theorem}
\cite[Lemma 3.4]{achterprieshyper}
Suppose $g \ge 2$ and $0 \le f \le g$.  Let $S$ be an irreducible
component of $\stk H_g^f$.
\begin{alphabetize}
\item If $f>0$, then each irreducible component of $\Delta_0[\bar S]$
  contains either an irreducible component of $\kappa_{g-1}(\barstk
  H_{g-1;1}^{f-1})$ or of some $\lambda_{i,g-1-i}(\barstk H_{i;1}^{f_1} \cross
  \barstk H_{g-1-i;1}^{f_2})$ with $0 \le f_1 \le i$, $0 \le f_2 \le
  g-1-i$, $f_1+f_2 = f$.
\item If $f=0$, then $\bar S$ contains the image of an irreducible
  component of $\tilstk H_i^0 \cross \tilstk H_{g-i}^0$ for some $1
  \le i \le g-1$.
\end{alphabetize}
\end{theorem}

In contrast to the case of general curves (Corollary \ref{cortree}), at
present one only knows that hyperelliptic curves degenerate to {\em
  trees} of elliptic curves:

\begin{corollary}
\cite[Theorem 3.11(c)] {achterprieshyper}
Suppose $g \ge 2$ and $ 0 \le f \le g$.  Let $S$ be an irreducible
component of $\stk H_g^f$.  Then $\bar S$ contains the moduli point of
some tree of elliptic curves, of which $f$ are ordinary and $g-f$ are
supersingular.
\end{corollary}

\paragraph{Stratification by $p$-rank}

To some extent, the $p$-rank stratification of $\stk H_g$ is known to be much like that of $\stk
M_g$.

Corollary \ref{PnestedMg} holds,
{\em mutatis mutandis},  for $\stk H_g$.  (Part (a) relies only on dimension
counting; part (b), whose proof relies on degeneration, already
appears as \cite[Corollary 3.15]{achterprieshyper}.)

However, even for $g=3$, we start losing information.
It is not known if $\stk H_3^0$ is irreducible.  While $\stk H_2^0$ is connected, it is not clear if $\tilstk H_2^0$ is connected, making it more difficult to prove a result analogous 
to Corollary \ref{thmgconn} for hyperelliptic curves.

\paragraph{Stratification by Newton polygon}

Less is known about the Newton stratification of $\stk H_g$ than
that of $\stk M_g$.      However, it {\em is}
known that, for each irreducible component of $\stk H_3^0$, the
generic Newton polygon is $\nu_3^0$ \cite{oorthess}.  
 
The analogue of Proposition \ref{Pinduction} is valid for $\stk H_g$,
too.    In part (a), even though one has less control over
degenerations, one still knows that, for an irreducible component $S$
of $\stk H_g^f$, $\bar S$ contains an irreducible component of
$\kappa_{1,g-1}(\tilstk H_1¹ \cross \tilstk H_{g-1}^{f-1})$. Part (b)
is valid for $\stk H_g$, as well. The key observation is that, if $S$
is an irreducible component of $\stk H_{g-1}^{f-1}$, then
$\kappa_{g-1}(S\cross_{\stk H_{g-1}} \stk H_{g-1;1})$ is in the boundary
of {\em some} irreducible component of $\stk H_g^f$.

We conclude that the generic Newton polygon of each component of $\stk
H_g^{g-3}$ is $\nu_g^{g-3}$.  

\section{Some conjectures about Newton polygons of curves}
\label{secconjectures}

In this section, we discuss variations of the following question.

\begin{question} \label{Qexists}
Given $g$ and $p$, does every symmetric admissible Newton polygon of
height $2g$ occur for the Jacobian of a smooth projective curve of genus $g$?
\end{question}

The first open case of this question is when $g=4$, for the Newton polygons $\sigma_4$, $\nu_3^0 \oplus \sigma_1$, and $\nu_1^1 \oplus \sigma_3$.  Using Theorem \ref{propmgboundary}, it is not hard to see that each of these Newton polygons is realized by a singular curve of compact type.  It is unlikely that appeal to Shimura varieties, as in the proof of Lemma \ref{lemg4}, will resolve this question for all $p$.  For example, if $p$ is inert in $\rat(\zeta_3)$, then no abelian variety with $p$-rank one has an action by $\integ[\zeta_3]$.  In particular, $\nu_1^1\oplus \sigma_3$ is not the Newton polygon of any abelian variety with moduli point in $\stk Z$.

\subsection{Non-existence philosophy}
\label{subsecnonexistence}

In \cite{problemsoncurves}, Oort explains why the answer to Question
\ref{Qexists} could be {\em no}. 
Consider the partial ordering of Newton polygons.
Suppose that:
\begin{description}
\item{(i)} $\codim(\xi) \ge 3g-3$, i.e., the length of the longest chain of Newton polygons connecting $\xi$ and $\nu_g^g$ is larger than $3g-3$; and
\item{(ii)} the denominators of $\xi$ are ``large''.
\end{description}
Then \cite[Expectation 8.5.4]{problemsoncurves} states that one expects that there is no curve of compact type whose Jacobian has Newton polygon $\xi$.

By this reasoning, one expects that there is no curve of genus $11$ whose Jacobian has Newton polygon $\xi=G_{5,6} \oplus G_{6,5}$, 
with slopes $5/11$ and $6/11$.  On the other hand, in characteristic $p=2$, Blache found a curve of genus $11$ over ${\mathbb F}_2$, namely
$y^2+y=x^{23}+x^{21}+x^{17}+x^7+x^5$, which does have slopes $5/11$ and $6/11$.

The dimension of $\cala_g$ is $g(g+1)/2$ and the dimension of its
supersingular locus is $\lfloor g^2/4 \rfloor$.  The length of the
longest chain of Newton polygons connecting the supersingular Newton
polygon $\sigma_g$ and the ordinary Newton polygon $\nu_g^g$ is the
difference between these, which is greater than $3g-3$ when $g \geq
9$.  It is still possible that every Newton polygon in the chain
occurs for a Jacobian, but if so, then there are two Newton polygons
$\xi_1$ and $\xi_2$, such that $\cala_g^{\xi_1}$ is in the closure of
$\cala_g^{\xi_2}$ but $\calm_g^{\xi_1}$ is not in the closure of
$\calm_g^{\xi_2}$.  In other words, under Condition (i), $\calm_g$ does
not admit a perfect stratification by Newton polygon.

The Newton polygon of a curve of compact type is the join of the Newton polygons of its components.
If the denominators of $\xi$ are all less than $g$, then one can try to construct a singular curve with Newton polygon 
$\xi$ from curves of smaller genus.  
For example, it is easy to see that $\sigma_g$ is the Newton polygon of a tree of supersingular elliptic curves.
If $\xi$ is indecomposable as a symmetric Newton polygon, then it cannot occur for the 
Jacobian of a singular curve of compact type.
As a means of making Condition (ii) more precise, one could restrict to the case that $\xi$ is indecomposable.

Another variation is to restrict to the Jacobians of smooth curves.
In fact, Oort conjectures that if $\xi_i$ is the Newton polygon of a point of $\calm_{g_i}$ for $i=1,2$, 
then the join of $\xi_1$ and $\xi_2$ occurs as the Newton polygon of a point of $\calm_{g_1+g_2}$, in other words, 
as the Newton polygon of the Jacobian of a {\it smooth} curve \cite[Conjecture 8.5.7]{problemsoncurves}.

\begin{remark}
The original motivation for this non-existence expectation was the following.
Let $Y^{(cu)}(K, g)$ denote the statement: There exists an abelian variety $A$ of dimension $g$ defined over $K$ 
which is not isogenous to the Jacobian of any curve of compact type;
(here, this is an isogeny of abelian varieties without polarization).
There is an expectation\footnote{attributed to Katz and to Oort by,
  respectively, Oort and Katz} that $Y^{(cu)}(\overline{\mathbb Q}, g)$ is true for every $g \geq 4$.
To prove $Y^{(cu)}(\overline{\mathbb Q}, g)$, it suffices to find a prime $p$ and a Newton polygon $\xi$ of height $2g$ 
such that $\xi$ does not occur as the Newton polygon of a Jacobian of a curve \cite[(8.5.1)]{problemsoncurves}.
It turns out that $Y^{(cu)}(\overline{\mathbb Q}, g)$ was proven by other methods \cite{chaioort12,tsimerman12}.
\end{remark}

\subsection{Supersingular curves}

Recall that $\sigma_g = \st{\half 1^{\oplus g}}$.  
The supersingular locus $\stk M_g^{\sigma_g}$ has 
been studied extensively. 
When $p=2$, Van der Geer and Van der Vlugt proved that there is a smooth curve of every genus which is supersingular.
More generally, the same methods show there exist supersingular curves of arbitrarily large genus  
defined over $\overline{\mathbb F}_p$ for all primes $p$.

\begin{example} 
Let $R(x)\in \bar\ff_p[x]$ be an additive polynomial, i.e., a polynomial of the form
$R[x]=a_0 + a_1 x^p + \cdots + a_d x^{p^d}$.
Consider the Artin-Schreier curve $Y$ with affine equation
$y^p-y=xR(x)$; it  has genus $(p-1)(p^d+1)/2$.
Then $Y$ is supersingular by \cite[Theorem 13.7]{vdgvdv}.   
\end{example}

\subsection{Other non-existence results}

To this date, the only non-existence results about Newton polygons are for Jacobians of curves with automorphisms.
Newton polygons of degree $p$ covers of ${\mathbb P}^1$ have been studied using techniques for exponential sums and Dwork cohomology.  
For example \cite{scholtenzhu}, when $p=2$, there are no hyperelliptic
curves of genus $2^n-1$ which are supersingular. 
More generally, there are conditions on the first slope of the Newton polygon of a degree $p$ cover of the projective line \cite{blachefirstslope}.

 \bibliographystyle{hamsplain}
\bibliography{hp}

\end{document}